\documentclass{amsart}
\usepackage{amsfonts,graphicx,rotating,slashbox,ctable}
\input diagxy

\newtheorem{theorem}{Theorem}[section]
\newtheorem{lemma}[theorem]{Lemma}

\newtheorem{proposition}[theorem]{Proposition}

\theoremstyle{definition}
\newtheorem{definition}[theorem]{Definition}

\theoremstyle{remark}
\newtheorem{remark}[theorem]{Remark}

\numberwithin{equation}{section}

\makeatletter
\def\imod#1{\allowbreak\mkern5mu({\operator@font mod}\,\,#1)}
\makeatother

\begin{document}
\title[Gr\"obner bases for (all) Grassmann  manifolds]
 {Gr\"obner bases for (all) Grassmann  manifolds}

\author{Zoran Z. Petrovi\'c}
\address{University of Belgrade,
  Faculty of mathematics,
  Studentski trg 16,
  Belgrade,
  Serbia}
\email{zoranp@matf.bg.ac.rs}
\thanks{The first author was partially supported by Ministry of Education, Science and Technological Development of Republic of Serbia Project \#174032.}

\author{Branislav I. Prvulovi\'c}
\address{University of Belgrade,
  Faculty of mathematics,
  Studentski trg 16,
  Belgrade,
  Serbia}
\email{bane@matf.bg.ac.rs}
\thanks{The second author was partially supported by Ministry of Education, Science and Technological Development of Republic of Serbia Project \#174034.}
\author{Marko Radovanovi\'c}
\address{University of Belgrade,
  Faculty of mathematics,
  Studentski trg 16,
  Belgrade,
  Serbia}
\email{markor@matf.bg.ac.rs}

\thanks{The third author was partially supported by Ministry of Education, Science and Technological Development of Republic of Serbia Project \#174008.}

\subjclass[2000]{13P10, 14M15, 57N65 (primary), 57R42, 57R20, 55S45 (secondary)}




\begin{abstract}
Grassmann manifolds $G_{k,n}$ are among the central objects in
geometry and topology. The Borel picture of the mod 2 cohomology of $G_{k,n}$ is
given as a polynomial algebra modulo a certain ideal $I_{k,n}$. The purpose of
this paper is to understand this cohomology via Gr\"obner
bases.
Reduced Gr\"obner bases for the ideals $I_{k,n}$ are determined. An
application of these bases is given by proving an immersion theorem for
Grassmann manifolds $G_{5,n}$, which establishes new immersions for
an infinite family of these manifolds.
\end{abstract}

\maketitle



\section{Introduction} 
\label{intro}
Mod 2 cohomology of Grassmann manifolds $G_{k,n}=O(n+k)/O(k)\times O(n)$ is the polynomial
algebra in Stiefel-Whitney classes $w_1,\dots,w_k$ of the canonical bundle over $G_{k,n}$ modulo the ideal
$I_{k,n}$ generated by dual classes $\overline{w}_{n+1},\dots,\overline{w}_{n+k}$. Although the description of this ideal is simple enough, concrete calculations in cohomology of Grassmann manifolds
may be rather difficult to perform. The question of whether a certain
cohomology class is zero is rather important in various applications ---
for example, in determining the span of Grassmannians, in discussing
immersions and embeddings in Euclidean spaces, in the determination of
cup-length (which is related to the Lusternik-Schnirelmann category), in
some
geometrical problems which may be reduced to the question of the
existence of a non-zero section of a bundle over a Grassmann manifold,
etc.  It is known that Gr\"obner bases are
useful when one works with polynomial algebras modulo certain ideal.
The first use of Gr\"obner bases in this context appears in \cite{Monks}
where the Gr\"obner bases for $I_{2,n}$ were established for $n$ of the
form $n=2^s-3$ and $n=2^s-4$. These bases were used to prove an
immersion result for corresponding Grassmann manifolds. Another
application of Gr\"obner bases in the similar context may be found in
\cite{Fukaya}.

In \cite{Petrovic} and \cite{Petrovic1} reduced Gr\"obner bases for
$I_{2,n}$ and $I_{3,n}$ (for all $n$) were established and used to obtain
some new immersion results for Grassmann manifolds. At the time of
writing of these papers, the authors were not aware of the paper
\cite{Jaworowski}, where additive bases for mod 2 cohomology of Grassmann
manifolds were established. These additive bases together with
information about Gr\"obner bases obtained directly in \cite{Petrovic} and
\cite{Petrovic1} allowed us to obtain reduced Gr\"obner bases for
all $I_{k,n}$.

The plan of the presentation is as follows. In Section 2 some
necessary facts about
cohomology algebra  $H^{*}(G_{k,n};\mathbb{Z}_{2})$ are reviewed.
Section 3 contains main results, namely, the determination of reduced
Gr\"obner bases for all $I_{k,n}$. Section 4 is devoted to an application
of the obtained results to the immersion problem for $G_{5,n}$ for
$n$ divisible by 8.

\section{The cohomology algebra $H^{*}(G_{k,n};\mathbb{Z}_{2})$}

In this section $n$ and $k$ are fixed integers such that $n\geq  k\geq 2$.
Let $G_{k,n}=G_{k}(\mathbb{R}^{n+k})$ be the Grassmann manifold of $k$-dimensional subspaces in $\mathbb{R}^{n+k}$. Let $\gamma_{k}$ be the canonical vector bundle over $G_{k,n}$ and $w_{1},w_{2},\ldots,w_{k}$ its Stiefel-Whitney classes.
It is a direct consequence of Borel's result (\cite{Borel}) that the mod $2$ cohomology algebra of $G_{k,n}$ is isomorphic to the polynomial algebra $\mathbb{Z}_{2}[w_{1},w_{2},\ldots,w_{k}]$ modulo the ideal $I_{k,n}$ generated by the dual classes $\overline{w}_{n+1},\overline{w}_{n+2},\ldots,\overline{w}_{n+k}$. The following equality holds for these dual classes:
\[(1+w_{1}+w_{2}+\cdots+w_{k})(1+\overline{w}_{1}+\overline{w}_{2}+\cdots)=1,\]
and therefore, they satisfy the recurrence relation
\begin{equation}\label{recurrence}
\overline{w}_{r+k}=\sum_{i=1}^{k}w_{i}\overline{w}_{r+k-i}, \quad r\geq 1.
\end{equation}
Also, it is not hard to verify that the explicit formula for $\overline{w}_{r}$ ($r\geq 1$) is the following (see \cite[p.\ 3]{Petrovic1}):
\begin{equation}\label{dualclass}
\overline{w}_{r}=\sum_{a_{1}+2a_{2}+\cdots+ka_{k}=r}[a_{1},a_{2},\ldots,a_{k}]w_{1}^{a_{1}}w_{2}^{a_{2}}\cdots w_{k}^{a_{k}},
\end{equation}
where $a_{1},a_{2},\ldots,a_{k}$ are understood to be nonnegative integers and $[a_{1},a_{2},\ldots,a_{k}]$ denotes the multinomial coefficient
\begin{eqnarray*}
[a_{1},a_{2},\ldots,a_{k}]
&=&\binom{a_{1}+a_{2}+\cdots+a_{k}}{a_{1}}\binom{a_{2}+\cdots+a_{k}}{a_{2}}\cdots\binom{a_{k-1}+a_{k}}{a_{k-1}}\\
&=&\prod_{t=2}^k\binom{\sum_{j=t-1}^k a_j}{a_{t-1}}.
\end{eqnarray*}

On the other hand, in \cite{Jaworowski} Jaworowski detected an additive basis for $H^{*}(G_{k,n};\mathbb{Z}_{2})$. Let us conclude this brief opening section by stating his result.
\begin{theorem}[\cite{Jaworowski}]\label{Jaworowski}
The set $B=\{w_{1}^{a_{1}}w_{2}^{a_{2}}\cdots w_{k}^{a_{k}}\mid a_{1}+a_{2}+\cdots+a_{k}\leq  n\}$ is a vector space basis for $H^{*}(G_{k,n};\mathbb{Z}_{2})\cong\mathbb{Z}_{2}[w_{1},w_{2},\ldots,w_{k}]/I_{k,n}$.
\end{theorem}

\section{Gr\"obner bases}

As usual, $\mathbb{Z}$ denotes the set of all integers. Recall that for $\alpha,\beta\in\mathbb Z$ the binomial coefficient $\binom{\alpha}{\beta}$ is defined by \[\binom{\alpha}{\beta}:=\left\{\begin{array}{cl}
                \frac{\alpha(\alpha-1)\cdots(\alpha-\beta+1)}{\beta!}, & \beta>0 \\
                1, & \beta=0 \\
                0, & \beta<0\end{array}\right.,\]
and therefore, the following lemma is straightforward.

\begin{lemma}\label{lem1}
If $\binom{\alpha}{\beta}\neq0$, then $\alpha\geq \beta$ or $\alpha\leq -1$.
\end{lemma}

Recall also the well-known formula (which holds for all $\alpha,\beta\in\mathbb Z$)
\begin{equation}\label{binomIdentity}
\binom{\alpha}{\beta}=\binom{\alpha-1}{\beta}+\binom{\alpha-1}{\beta-1}.
\end{equation}

Let us now introduce some notations that we are going to use throughout this section. For an integer $m\geq  2$ and an $m$-tuple $N$ of integers we define the following $m$-tuples obtained from $N$ (for $i\leq  m$ and $i<j\leq  m$):
\begin{itemize}
\item[] $N^i$ -- by adding 1 to the $i$-th coordinate of $N$ (if $i<1$, then $N^i:=N$);
\item[] $N_{i}$ -- by subtracting 1 from the $i$-th coordinate of $N$ (if $i<1$, then $N_i:=N$);
\item[] $N^{i,j}$ -- by adding 1 to the $i$-th and $j$-th coordinate of $N$ (if $i<1$, then $N^{i,j}:=N^j$);
\item[] $N^{i,i}$ -- by adding 2 to the $i$-th coordinate of $N$ (if $i<1$, then $N^{i,i}:=N$);
\item[] $N_{i,j}$ -- by subtracting 1 from the $i$-th and $j$-th coordinate of $N$ (if $i<1$, then $N_{i,j}:=N_j$);
\item[] $N_{i,i}$ -- by subtracting 2 from the $i$-th coordinate of $N$ (if $i<1$, then $N_{i,i}:=N$).
\end{itemize}
For an integer $k\geq 2$, a $k$-tuple $A=(a_1,a_2,\ldots,a_k)$ and a $(k-1)$-tuple $M=(m_2,m_3,\ldots,m_{k})$ of  integers, let:
\begin{itemize}
\item[] $S_A:=\displaystyle\sum_{j=1}^k a_j$, $S'_A:=\displaystyle\sum_{j=1}^k ja_j$, and $S_M:=\displaystyle\sum_{j=2}^k m_j$, $S'_M:=\displaystyle\sum_{j=2}^k (j-1)m_j$;

\item[] $P_t(A,M):=\displaystyle\binom{\sum_{j=t-1}^k a_j-\sum_{j=t}^k m_j}{a_{t-1}}$, for $t=\overline{2,k}$;

\item[] $P(A,M):=\displaystyle\prod_{t=2}^k P_t(A,M)$.
\end{itemize}
For example, $P_{2}(A,M)=\displaystyle\binom{S_{A}-S_{M}}{a_{1}}$. Also, $P(A,\mathbf{0})=[a_{1},a_{2},\ldots,a_{k}]$, where $\mathbf{0}=(\displaystyle\underbrace{0,0,\ldots,0}_{k-1})$.

Henceforth, the integers $k$ and $n$ with the property $n\geq  k\geq 2$ are fixed. Observe the polynomial algebra $\mathbb{Z}_{2}[w_{1},w_{2},\ldots,w_{k}]$. Let us now define certain polynomials in $\mathbb{Z}_{2}[w_{1},w_{2},\ldots,w_{k}]$ which will be important in our considerations.

\begin{definition}\label{defG}  For a $(k-1)$-tuple of nonnegative integers $M=(m_2,\ldots,m_k)$, let
\[g_M:=\sum_{S'_A=n+1+S'_M}P(A,M)\cdot W^A,\]
where the sum is taken over all $k$-tuples of nonnegative integers $A=(a_1,a_2,\ldots,a_k)$ such that $S'_A=n+1+S'_M$, and $W^A=w_1^{a_1}w_2^{a_2}\cdots w_k^{a_k}$.

Moreover, let\[G:=\{g_M\mid S_M\leq  n+1\}.\]
\end{definition}
Note that, by (\ref{dualclass}), $\overline{w}_{n+1}=g_{\mathbf{0}}\in G$.

Our aim is to prove that $G$ is a Gr\"obner basis for  $I_{k,n}=(\overline{w}_{n+1},,\ldots,\overline{w}_{n+k})$ which determines the cohomology algebra $H^{*}(G_{k,n};\mathbb{Z}_{2})$. In order to do so, first we need to specify a term ordering in $\mathbb{Z}_{2}[w_{1},w_{2},\ldots,w_{k}]$. We shall use the grlex ordering (which will be denoted by $\preceq $) on terms (monomials) in $\mathbb{Z}_{2}[w_{1},w_{2},\ldots,w_{k}]$ with $w_{1}>w_{2}>\cdots>w_{k}$. It is defined as follows. The terms are compared by the sum of the exponents and if these are equal for two terms, they are compared lexicographically from the left. That is, for $k$-tuples $A$ and $B$ of nonnegative integers we shall write $W^{A}\prec W^{B}$ if either $S_{A}<S_{B}$ or else $S_{A}=S_{B}$ and $a_{s}<b_{s}$ where $s=\min \{i\mid a_{i}\neq b_{i}\}$. Of course, $W^{A}\preceq  W^{B}$ means that either $W^{A}\prec W^{B}$ or $W^{A}=W^{B}$.

In fact, we are going to prove that $G$ is the reduced Gr\"obner basis for $I_{k,n}$ with respect to the grlex ordering $\preceq $. We start with a lemma.

\begin{lemma}\label{binom0}
If a $k$-tuple $A=(a_1,\ldots,a_k)$ and a $(k-1)$-tuple $M=(m_2,\ldots,m_k)$ of nonnegative integers are such that $P(A,M)\neq0$, then $S_A<S_M$ or else $\sum_{j=t}^k a_j\geq \sum_{j=t}^k m_j$ for all $t=\overline{2,k}$.
\end{lemma}

\begin{proof}
Let us assume that $S_A\geq  S_M$. We will prove by induction on $t$ that $\sum_{j=t}^k a_j\geq \sum_{j=t}^k m_j$, for $t=\overline{2,k}$. Since
$\binom{S_A-S_M}{a_1}=P_2(A,M)\neq0$ and $S_A-S_M\geq 0$,
by Lemma \ref{lem1} we have that $S_A-S_M\geq  a_1$, and therefore $\sum_{j=2}^k a_j\geq \sum_{j=2}^k m_j$.

Suppose now that $\sum_{j=t}^k a_j\geq \sum_{j=t}^k m_j$ for some $t$ such that $2\leq  t\leq  k-1$. Since $P_{t+1}(A,M)\neq0$ and $\sum_{j=t}^k a_j\geq \sum_{j=t}^k m_j\geq \sum_{j=t+1}^k m_j$, again by Lemma \ref{lem1} we conclude that $\sum_{j=t}^k a_j-\sum_{j=t+1}^k m_j\geq  a_t$. Hence, $\sum_{j=t+1}^k a_j\geq \sum_{j=t+1}^k m_j$.
\end{proof}

For a nonzero polynomial $f=\sum_{i=1}^{r}t_{i}\in\mathbb{Z}_{2}[w_{1},w_{2},\ldots,w_{k}]$, where $t_{i}$ are pairwise different terms, let $T(f):=\{t_1,t_2,\ldots,t_r\}$ ($T(0):=\emptyset$). The leading term of $f\neq0$, denoted by $\mathrm{LT}(f)$, is defined as $\max T(f)$ with respect to $\preceq $.

\begin{proposition}\label{degree}
Let $M=(m_2,\ldots,m_k)$ be a $(k-1)$-tuple of nonnegative integers such that $S_{M}\leq  n+1$ (i.e., such that $g_M\in G$). Then $g_M\neq 0$ and $\mathrm{LT}(g_M)=W^{\overline{M}}$, where \linebreak $\overline{M}=(n+1-S_M,m_2,\ldots,m_k)$.
Moreover, if $W^{A}\in T(g_M)\setminus\{W^{\overline{M}}\}$ for some $k$-tuple $A$ of nonnegative integers, then $S_{A}<n+1$.
\end{proposition}

\begin{proof}
If we define $m_1:=n+1-S_M$, then obviously $P_t(\overline{M},M)=\binom{m_{t-1}}{m_{t-1}}=1$, for $t=\overline{2,k}$, and therefore $P(\overline{M},M)=1$. Furthermore,\[S'_{\overline{M}}=\sum_{j=1}^k jm_j=n+1-S_M+\sum_{j=2}^k jm_j=n+1+\sum_{j=2}^k (j-1) m_j=n+1+S'_M,\]
and hence $W^{\overline{M}}\in T(g_M)$. So, $g_M\neq0$.

Now take a $k$-tuple $A=(a_1,\ldots,a_k)$ of nonnegative integers such that $S'_A=n+1+S'_M$ and $P(A,M)\equiv1\pmod2$, i.e., $W^A\in T(g_M)$. Since $S_{\overline{M}}=n+1$, in order to finish the proof of the proposition, it suffices to show that if $S_A\geq  n+1$, then $A=\overline{M}$.

Since $S_M\leq  n+1\leq  S_A$, by Lemma \ref{binom0} we have the following $k-1$ inequalities:
\begin{equation}\label{ineq}
\begin{array}{rcl}
                a_k&\geq & m_k, \\
                a_{k-1}+a_k&\geq & m_{k-1}+m_k,\\
                &\vdots&\\
                a_2+\cdots+a_k&\geq &m_2+\cdots+m_k.
\end{array}
\end{equation}
Summing up these inequalities, we get
\[\sum_{j=2}^k (j-1)a_j\geq  \sum_{j=2}^k (j-1)m_j.\]

On the other hand, since $S_A\geq  n+1$ and $S'_A=n+1+S'_M$,
\[\sum_{j=2}^k (j-1)a_j=\sum_{j=1}^k (j-1)a_j=S'_A-S_A\leq  S'_A-(n+1)=S'_M=\sum_{j=2}^k (j-1)m_j,\]
so all the inequalities in (\ref{ineq}) are in fact equalities and $S_A=n+1$. Hence, $a_t=m_t$ for $t=\overline{2,k}$, and $a_1=S_{A}-\sum_{j=2}^{k}a_j=n+1-S_M$, i.e., $A=\overline{M}$.
\end{proof}

Prior to the formulation of the following lemma, we would like to emphasize that for a $(k-1)$--tuple $M=(m_2,\ldots,m_k)$, by our definition, $M^i=(m_2,\ldots,m_{i+1}+1,\ldots,m_k)$, $i=\overline{1,k-1}$, and likewise for $M^{i,j}$, $M^{i,i}$, $M_i$, etc. For example, the $(k-1)$--tuple $M^{2}$ is defined as $(m_2,m_3+1,\ldots,m_k)$, and not as $(m_2+1,m_3,\ldots,m_k)$.

\begin{lemma}\label{pab} Let $A=(a_1,a_2,\ldots,a_k)$ be a $k$-tuple and $M=(m_2,\ldots,m_k)$ a $(k-1)$-tuple of integers.
\begin{itemize}
\item[(a)] For $1\leq  i\leq  j\leq  k-2$,
          \[P(A,M^{i,j})\equiv P(A_{i},M^j)+P(A,M^{i-1,j+1})+P(A_{j+1},M^{i-1})\pmod2.\]
\item[(b)] For $1\leq  i\leq  k-1$,
          \[P(A,M^{i,k-1})\equiv P(A_i,M^{k-1})+P(A_k,M^{i-1})\pmod2.\]
\end{itemize}
\end{lemma}

\begin{proof} Let $1\leq  i\leq  j\leq  k-1$. It is immediate from the definition that for all $t=\overline{2,k}$,
\[P_t(A,M^{i,j})=\binom{a_{t-1}+a_t+\cdots+a_k-m_t-\cdots-m_k-\delta_t}{a_{t-1}},\]
where $\delta_t=\left\{\begin{array}{rl}
                2, & t\leq  i+1 \\
                1, & i+2\leq  t\leq  j+1 \\
                0, & t>j+1\end{array}\right.$.
Also, if $t\neq i+1$, then
\[P_t(A_i,M^j)=\binom{a_{t-1}+a_t+\cdots+a_k-m_t-\cdots-m_k-\delta_t}{a_{t-1}},\]
 and so,
\begin{equation}
P_t(A,M^{i,j})=P_t(A_i,M^j), \quad\mbox{ for } t\neq i+1.\label{1.1}
\end{equation}

Likewise, using formula (\ref{binomIdentity}) we get
\begin{equation}
P_{i+1}(A,M^{i,j})+P_{i+1}(A_i,M^j)=P_{i+1}(A,M^j),\label{1.2}
\end{equation}
since the left-hand side is equal to \[\binom{a_i+\cdots+a_k-m_{i+1}-\cdots-m_k-2}{a_i}+\binom{a_i+\cdots+a_k-m_{i+1}-\cdots-m_k-2}{a_i-1}\]
and the right-hand side to
\[\binom{a_i+\cdots+a_k-m_{i+1}-\cdots-m_k-1}{a_i}.\]

\medskip

(a) In this case, similarly as for (\ref{1.1}) and (\ref{1.2}), one obtains the following equalities:
\begin{eqnarray}
P_t(A_i,M^j)&=&P_t(A,M^{i-1,j+1}),\quad\mbox{for $t\not\in\{i+1,j+2\}$}\\
P_t(A,M^{i-1,j+1})&=&P_t(A_{j+1},M^{i-1}),\quad\mbox{for $t\neq j+2$}\\
P_{i+1}(A,M^j)&=&P_{i+1}(A,M^{i-1,j+1})\\
P_{j+2}(A_i,M^j)&=& P_{j+2}(A,M^{i-1,j+1})+P_{j+2}(A_{j+1},M^{i-1}).\label{2.4}
\end{eqnarray}
So, using identities (\ref{1.1})--(\ref{2.4}), we have
\begin{align*}
\phantom{a}&P(A,M^{i,j})=\prod_{t=2}^k P_t(A,M^{i,j})=P_{i+1}(A,M^{i,j})\cdot\prod_{\substack{t=2\\ t\neq i+1}}^{k}P_t(A_i,M^j)\\
&\equiv\left(P_{i+1}(A_i,M^j)+P_{i+1}(A,M^j)\right)\cdot\prod_{\substack{t=2\\ t\neq i+1}}^{k}P_t(A_i,M^j) \\
&=\prod_{t=2}^{k}P_t(A_i,M^j)+P_{i+1}(A,M^{i-1,j+1})\cdot\prod_{\substack{t=2\\ t\neq i+1}}^{k}P_t(A_i,M^j)\\
&=P(A_i,M^j)+P_{j+2}(A_i,M^j)\cdot\prod_{\substack{t=2\\ t\neq j+2}}^{k}P_t(A,M^{i-1,j+1})\\
&=P(A_i,M^j)+\left(P_{j+2}(A,M^{i-1,j+1})+P_{j+2}(A_{j+1},M^{i-1})\right)\prod_{\substack{t=2\\t\neq j+2}}^{k}P_t(A,M^{i-1,j+1}) \\
&=P(A_i,M^j)+P(A,M^{i-1,j+1})+P(A_{j+1},M^{i-1})\pmod2.
\end{align*}

\medskip

(b) In a similar manner as before, for $1\leq  i\leq  k-1$ one can obtain two additional equalities:
\begin{eqnarray}
P_t(A_i,M^{k-1})&=&P_t(A_k,M^{i-1}),\quad\mbox{for $t\neq i+1$},\label{3.1}\\
P_{i+1}(A,M^{k-1})&=&P_{i+1}(A_k,M^{i-1}).\label{3.2}
\end{eqnarray}
Now, using identities (\ref{1.1})--(\ref{1.2}) and (\ref{3.1})--(\ref{3.2}), we have
\begin{eqnarray*}
P(A,M^{i,k-1})&=&\prod_{t=2}^{k}P_t(A,M^{i,k-1})=P_{i+1}(A,M^{i,k-1})\cdot\prod_{\substack{t=2\\t\neq i+1}}^{k}P_{t}(A_i,M^{k-1})\\
&\equiv&\left(P_{i+1}(A_i,M^{k-1})+P_{i+1}(A,M^{k-1})\right)\cdot\prod_{\substack{t=2\\t\neq i+1}}^k P_{t}(A_i,M^{k-1})\\
&=&P(A_i,M^{k-1})+P(A_k,M^{i-1})\pmod2,
\end{eqnarray*}
and we are done.
\end{proof}

Note that we could unify parts (a) and (b) of the previous lemma by stating that $P(A,M^{i,j})\equiv P(A_{i},M^j)+P(A_{j+1},M^{i-1})+P(A,M^{i-1,j+1})\pmod2$, for $1\leq  i\leq  j\leq  k-1$, with convention that $P(A,M^{i-1,j+1})=0$ if $j=k-1$.

\begin{proposition}\label{reccurence} Let $M=(m_2,\ldots,m_k)$ be a $(k-1)$-tuple of nonnegative integers and $1\leq  i\leq  j\leq  k-1$. Then in the polynomial algebra $\mathbb{Z}_{2}[w_{1},w_{2},\ldots,w_{k}]$, we have
the identity
\[g_{M^{i,j}}=w_ig_{M^j}+w_{j+1}g_{M^{i-1}}+g_{M^{i-1,j+1}},\]
where the polynomial $g_{M^{i-1,j+1}}$ is understood to be zero if $j=k-1$.
\end{proposition}

\begin{proof}
By Lemma \ref{pab} we have

\begin{eqnarray*}
g_{M^{i,j}}&=&\sum_{S'_A=n+1+S'_{M^{i,j}}}P(A,M^{i,j})\cdot W^A\\
&=&\sum_{S'_A=n+1+S'_{M^{i,j}}}\left(P(A_{i},M^j)+P(A_{j+1},M^{i-1})+P(A,M^{i-1,j+1})\right)\cdot W^A\\
&=&\sum_{S'_A=n+1+S'_{M^{i,j}}}P(A_{i},M^j)\cdot W^{A}+\sum_{S'_A=n+1+S'_{M^{i,j}}}P(A_{j+1},M^{i-1})\cdot W^{A}\\
&\phantom{a}&+g_{M^{i-1,j+1}},
\end{eqnarray*}
since $S'_{M^{i,j}}=S'_M+i+j=S'_M+i-1+j+1=S'_{M^{i-1,j+1}}$ (for $j\leq  k-2$). Observe also that the equality $S'_A=n+1+S'_{M^{i,j}}$ is equivalent to $S'_{A_{i}}=S'_A-i=n+1+S'_{M^{i,j}}-i=n+1+S'_M+j=n+1+S'_{M^{j}}$, and likewise, it is equivalent to $S'_{A_{j+1}}=n+1+S'_{M^{i-1}}$.

Now, consider the first sum in the upper expression. Since the sum is taken over the $k$-tuples $A=(a_1,a_2,\ldots,a_k)$ of nonnegative integers (such that $S'_A=n+1+S'_{M^{i,j}}$), the coordinates of $A_i$ are also nonnegative with exception that its $i$-th coordinate might be $-1$ (if $a_i=0$). But, in that case, $P_{i+1}(A_i,M^j)=\binom{a_{i+1}+\cdots+a_k-m_{i+1}-\cdots-m_k-2}{-1}=0$, and so $P(A_i,M^j)=0$. Therefore, we may assume that $a_i\geq 1$, and consequently, that $A_i$ runs through the set of $k$-tuples of nonnegative integers (such that $S'_{A_{i}}=n+1+S'_{M^{j}}$). Hence,
\[\sum_{S'_A=n+1+S'_{M^{i,j}}}P(A_{i},M^j)\cdot W^{A}=w_{i}\sum_{S'_{A_{i}}=n+1+S'_{M^{j}}}P(A_{i},M^j)\cdot W^{A_i}=w_ig_{M^j}.\]

So, we are left to prove that the second sum in the upper expression for $g_{M^{i,j}}$ is equal to $w_{j+1}g_{M^{i-1}}$. Let $A=(a_1,a_2,\ldots,a_k)$ be a $k$-tuple of nonnegative integers such that $S'_A=n+1+S'_{M^{i,j}}$, i.e., $S'_{A_{j+1}}=n+1+S'_{M^{i-1}}$. It suffices to show that $a_{j+1}=0$ implies $P(A_{j+1},M^{i-1})=0$, since then the proof follows as for the first sum.

If $j+1<k$, then $a_{j+1}=0$ implies $P_{j+2}(A_{j+1},M^{i-1})=0$, and so, $P(A_{j+1},M^{i-1})=0$.

For $j=k-1$, let us assume to the contrary that $a_k=0$ and $P(A_{k},M^{i-1})\neq0$. First we shall prove that
\begin{equation}
a_{t-1}+a_{t}+\cdots+a_{k-1}\leq  m_t+\cdots+m_k+\varepsilon_{t}, \quad\mbox{ for all $t=\overline{2,k}$,}\label{induction}
\end{equation}
where $\varepsilon_{t}=\left\{\begin{array}{rl}
                1, & 2\leq  t\leq  i \\
                0, & i+1\leq  t\leq  k\end{array}\right.$.
The proof is by reverse induction on $t$. For the induction base we prove (\ref{induction}) for $t=k$. Since $\binom{a_{k-1}-1-m_k}{a_{k-1}}=P_k(A_{k},M^{i-1})\neq0$ and $a_{k-1}-1-m_k<a_{k-1}$, by Lemma \ref{lem1} we conclude that $a_{k-1}-1-m_k\leq -1$, so $a_{k-1}\leq  m_k=m_k+\varepsilon_{k}$. For the inductive step, let $2\leq  t\leq  k-1$, and suppose that $a_{t}+\cdots+a_{k-1}\leq  m_{t+1}+\cdots+m_k+\varepsilon_{t+1}$. Since obviously $\varepsilon_{t+1}\leq \varepsilon_t$, we actually have that $a_{t}+\cdots+a_{k-1}\leq  m_{t+1}+\cdots+m_k+\varepsilon_{t}$. Since
\[P_t(A_{k},M^{i-1})=\binom{a_{t-1}+a_{t}+\cdots+a_{k-1}-1-m_t- m_{t+1}-\cdots-m_k-\varepsilon_t}{a_{t-1}}\neq0,\]
and $a_{t-1}+a_{t}+\cdots+a_{k-1}-1-m_t- m_{t+1}-\cdots-m_k-\varepsilon_t\leq  a_{t-1}-1-m_t<a_{t-1}$, according to Lemma \ref{lem1}, we have that $a_{t-1}+a_{t}+\cdots+a_{k-1}-1-m_t- m_{t+1}-\cdots-m_k-\varepsilon_t\leq -1$, i.e., $a_{t-1}+a_{t}+\cdots+a_{k-1}\leq  m_t+m_{t+1}+\cdots+m_k+\varepsilon_{t}$.

Now, summing up inequalities (\ref{induction}), we get
\[S'_A\leq  S'_M+\sum_{t=2}^{k}\varepsilon_t<S'_M+k-1<S'_{M}+n+1\leq  S'_{M^{i-1}}+n+1=S'_{A_k},\]
which is a contradiction since $S'_A>S'_A-k=S'_{A_k}$.
\end{proof}

In the following lemma we establish a connection between polynomials $g_{M}$ and polynomials (dual classes) $\overline{w}_{r}\in\mathbb{Z}_{2}[w_{1},w_{2},\ldots,w_{k}]$ from the previous section.

\begin{lemma}\label{m00}
For $m\geq 0$ and $(k-1)$-tuple $M=(m,0,\ldots,0)$ we have that
\[g_M=\sum_{i=0}^m\binom{m}{i}w_1^{m-i}\overline{w}_{n+1+i}.\]
\end{lemma}
\begin{proof} The polynomials $g_M$ were introduced in Definition \ref{defG} and they depend on the (previously fixed) integer $n$. In this proof (and only in this proof) we allow $n$ to vary through the set $\{k,k+1,\ldots\}$, while the integer $k\geq 2$ is still fixed (we are working in the polynomial algebra $\mathbb{Z}_{2}[w_{1},w_{2},\ldots,w_{k}]$). Note that the polynomials $\overline{w}_{r}$, $r\geq 1$, are defined independently of $n$. We emphasize the dependence of $g_M$ on $n$ by using an appropriate superscript, and we  actually prove the following claim:
\[g^{(n)}_M=\sum_{i=0}^m\binom{m}{i}w_1^{m-i}\overline{w}_{n+1+i}, \quad\mbox{ for all } m\geq 0 \mbox{ and all } n\geq  k.\]
The proof is by induction on $m$. We have already noticed that $g_{\mathbf{0}}^{(n)}=\overline{w}_{n+1}$, and therefore, the claim is true for $m=0$ (and all $n\geq  k$). So, let $m\geq 1$ and assume that the claim is true for the integer $m-1$ and all $n\geq  k$. Let $M=(m,0,\ldots,0)$ and $n\geq  k$. Then $M_1=(m-1,0,\ldots,0)$ and since, for all $k$-tuples $A$ of integers, $P_2(A,M)\equiv P_2(A_1,M_1)+P_2(A,M_1)\pmod2$ by (\ref{binomIdentity}) and $P_t(A,M)=P_t(A_1,M_1)=P_t(A,M_1)$ for $t=\overline{3,k}$, we have that
\begin{eqnarray*}
g^{(n)}_M&=&\sum_{S'_A=n+1+S'_M}P(A,M)\cdot W^A\\
&=&\sum_{S'_A=n+1+S'_M}\left(\left(P_2(A_1,M_1)+P_2(A,M_1)\right)\cdot\prod_{t=3}^k P_t(A,M)\right)\cdot W^A\\
&=&w_1\sum_{S'_{A_1}=n+1+S'_{M_1}}P(A_1,M_1)\cdot W^{A_1}+\sum_{S'_A=(n+1)+1+S'_{M_1}}P(A,M_1)\cdot W^A\\
&=&w_1g^{(n)}_{M_1}+g^{(n+1)}_{M_1}\\
&=&w_1\sum_{i=0}^{m-1}\binom{m-1}{i}w_1^{m-1-i}\overline{w}_{n+1+i} +\sum_{i=0}^{m-1}\binom{m-1}{i}w_1^{m-1-i}\overline{w}_{n+2+i}\\
&=&\sum_{i=0}^{m}\binom{m-1}{i}w_1^{m-i}\overline{w}_{n+1+i}+\sum_{i=0}^{m}\binom{m-1}{i-1}w_1^{m-i}\overline{w}_{n+1+i} \\
&=&\sum_{i=0}^m\binom{m}{i}w_1^{m-i}\overline{w}_{n+1+i},
\end{eqnarray*}
and the proof is completed.
\end{proof}

\begin{proposition}\label{subset}
$G\subseteq I_{k,n}$.
\end{proposition}
\begin{proof} Since the ideal $I_{k,n}$ is generated by the polynomials $\overline{w}_{n+1},\overline{w}_{n+2},\ldots,\overline{w}_{n+k}$, note that, by the recurrence relation (\ref{recurrence}), not only these $k$ polynomials, but all $\overline{w}_{r}$ for $r\geq  n+1$ belong to $I_{k,n}$. Likewise, we shall prove that $g_M\in I_{k,n}$ for all $(k-1)$-tuples $M$ of nonnegative integers, and not only for those with the property $S_M\leq  n+1$ (i.e., $g_M\in G$).

We define the relation $<_{lexr}$ on the set of all $(k-1)$-tuples of nonnegative integers by \[(a_1,a_2,\ldots,a_{k-1})<_{lexr}(b_1,b_2,\ldots,b_{k-1})\Longleftrightarrow a_s<b_s,\mbox{ where }s=\max\{i\mid a_i\neq b_i\},\]
which is exactly the strict part of the lexicographical right ordering. This is a well ordering and our proof is by induction on $<_{lexr}$.

For the $(k-1)$-tuple $M=(m,0,\ldots,0)$, where $m\geq 0$ is arbitrary integer, from Lemma \ref{m00} and our remark at the beginning of this proof, we immediately get that $g_M\in I_{k,n}$. So, let us now take a $(k-1)$-tuple $M=(m_{2},m_{3},\ldots,m_k)$ such that the greatest integer $s$ with the property $m_{s+1}>0$ is at least $2$. Hence, $2\leq  s\leq  k-1$ and $M=(m_{2},\ldots,m_{s+1},0,\ldots,0)$. Let us also assume that $g_{M'}\in I_{k,n}$ for all $M'$ such that $M'<_{lexr} M$. We wish to prove that $g_M\in I_{k,n}$. By Proposition \ref{reccurence} applied to the $(k-1)$-tuple $M_s$, $i=1$ and $j=s-1$,
\[g_M=g_{M^{1,s-1}_s}+w_1g_{M^{s-1}_s}+w_{s}g_{M_s}.\]
Since $M_s<_{lexr} M^{s-1}_s<_{lexr} M^{1,s-1}_s<_{lexr} M$, we conclude that $g_M\in I_{k,n}$.
\end{proof}

In the following proposition we formulate a characterization of Gr\"obner bases which we shall use for the proof that the set $G$ is a Gr\"obner basis for the ideal $I_{k,n}$. The proof of the proposition can be found in \cite[Proposition 5.38(vi)]{Becker}.

\begin{proposition}\label{beck}
Let $\mathbb{F}$ be a field, $\mathbb{F}[x_1,x_2,\ldots,x_k]$ the polynomial algebra, and $I$ an ideal in $\mathbb{F}[x_1,x_2,\ldots,x_k]$. Suppose that a term ordering $\preceq $ in $\mathbb{F}[x_1,x_2,\ldots,x_k]$ is fixed. Let $G$ be a finite subset of $I$ such that $0\notin G$, and let $B\subseteq\mathbb{F}[x_1,x_2,\ldots,x_k]/I$ be the set of cosets of all terms which are not divisible by any of the leading terms $\mathrm{LT}(g)$, $g\in G$. Then $G$ is a Gr\"obner basis for $I$ with respect to $\preceq $ if and only if $B$ is a vector space basis for $\mathbb{F}[x_1,x_2,\ldots,x_k]/I$.
\end{proposition}

We are now finally in position to prove our main result.

\begin{theorem}\label{main}
The set $G$ (see Definition \ref{defG}) is the reduced Gr\"obner basis for the ideal $I_{k,n}$ with respect to the grlex ordering $\preceq $.
\end{theorem}

\begin{proof} By Propositions \ref{degree} and \ref{subset}, $0\notin G\subseteq I_{k,n}$, and it is obvious from the definition that $G$ is finite. Also, according to Proposition \ref{degree} again, the set $\{\mathrm{LT}(g)\mid g\in G\}$ is exactly the set of all terms in $\mathbb{Z}_2[w_1,w_2,\ldots,w_k]$ with the sum of the exponents equal to $n+1$, that is $\{\mathrm{LT}(g)\mid g\in G\}=\{W^A\mid S_A=n+1\}$. Therefore, the set of all terms which are not divisible by any of the terms in $\{\mathrm{LT}(g)\mid g\in G\}$ is just the set $\{W^A\mid S_A\leq  n\}$. By Proposition \ref{beck} and Theorem \ref{Jaworowski}, $G$ is a Gr\"obner basis for $I_{k,n}$.

Since $\{\mathrm{LT}(g)\mid g\in G\}=\{W^A\mid S_A=n+1\}$ and all terms of $g\in G$ except the leading one have the sum of the exponents at most $n$ (Proposition \ref{degree}), no term of $g$ is divisible by any other leading term in $G$. This means that Gr\"obner basis $G$ is the reduced one.
\end{proof}

Propositions \ref{degree} and \ref{reccurence} enable us to explicitly determine the polynomials $g_M\in G$ for the $(k-1)$-tuples $M=(m_2,m_3,\ldots,m_k)$ such that $m_k$ is close to $n$. Namely, if $g_M\in G$ and $W^A\in T(g_M)\setminus\{W^{\overline{M}}\}$ (where $\overline{M}=(n+1-S_M,m_2,\ldots,m_k)$), then $S_A\leq  n$ by Proposition \ref{degree}. Consequently, $S'_A=\sum_{j=1}^kja_j\leq  k\sum_{j=1}^ka_j=kS_A\leq  kn$. On the other hand, $S'_A=n+1+S'_M$, and so, we conclude that $g_M=W^{\overline{M}}$ whenever $S'_M>(k-1)n-1$.

Let $N$ be the $(k-1)$-tuple $(0,\ldots,0,n)$. Since $S'_{N^s}>S'_N=(k-1)n$ (for $s=\overline{1,k-1}$), by the previous remark we have that
\begin{equation}\label{00n}
g_{N}=w_1 w_k^{n}\quad \mbox{ and }\quad
g_{N^{s}}=w_{s+1}w_k^n,\quad 1\leq  s\leq  k-1.
\end{equation}
If we apply Proposition \ref{reccurence} to the $(k-1)$-tuple $N_{k-1}=(0,\ldots,0,n-1)$, $i=1$ and $j=k-1$, we obtain the relation $w_kg_{N_{k-1}}=g_{N^{1}}+w_1g_{N}$. Both summands on the right-hand side contain $w_k$ as a factor, so $w_k$ cancels out and using (\ref{00n}) we get
\begin{equation}\label{00n-1}
g_{N_{k-1}}=w_1^2 w_k^{n-1}+w_2 w_k^{n-1}.
\end{equation}
Likewise, by applying Proposition \ref{reccurence} to $N_{k-1}$, $i=s+1$ and $j=k-1$, one obtains that $w_kg_{N^s_{k-1}}=w_{s+1}g_{N}+g_{N^{s+1}}$, and so
\begin{equation}\label{010n-1}
g_{N^s_{k-1}}=w_1w_{s+1} w_k^{n-1}+w_{s+2}w_{k}^{n-1},\quad 1\leq  s\leq  k-2.
\end{equation}
Identities (\ref{00n-1}) and (\ref{010n-1}) determine $g_M\in G$ when $m_k=n-1$ and $S_M\leq  n$. For computing $g_M\in G$ when $m_k=n-1$ and $S_M=n+1$ for a concrete integer $k$, one can use Proposition \ref{reccurence} and apply it first to $N_{k-1}$, $i=1$ and all $j=\overline{1,k-2}$, then to $N_{k-1}$, $i=2$ and all $j=\overline{2,k-2}$ and so on. After that, the polynomials $g_M\in G$ for $m_k=n-2$ can be obtained in the same manner -- by suitable applications of Proposition \ref{reccurence}. Actually, in the cases $k=2$ and $k=3$, all the members $g_M$ of Gr\"obner basis $G$ for $m_k\geq  n-5$ were listed in \cite{Petrovic} (for $k=2$) and \cite{Petrovic1} (for $k=3$).

Since our application of Gr\"obner bases (given in the next section) treats the case $k=5$, let us write down the relations (\ref{00n-1}) and (\ref{010n-1}) in this case:
\begin{equation}\label{k=5}
\begin{array}{rcl}
                g_{(0,0,0,n-1)}&=&w_1^2w_5^{n-1}+w_2w_5^{n-1}, \\
                g_{(1,0,0,n-1)}&=&w_1w_2w_5^{n-1}+w_3w_5^{n-1}, \\
                g_{(0,1,0,n-1)}&=&w_1w_3w_5^{n-1}+w_4w_5^{n-1}, \\
                g_{(0,0,1,n-1)}&=&w_1w_4w_5^{n-1}+w_5^{n}. \\
\end{array}
\end{equation}

\begin{remark}
Since the description of the mod $2$ cohomology of the complex Grassmann manifolds $G_k(\mathbb{C}^{n+k})$ is essentially the same as the one in the real case (the only difference being in the fact that dimensions of the Stiefel-Whitney classes are multiplied by $2$), it is immediate that the reduced Gr\"obner basis for the corresponding ideal in this case can be obtained from the Gr\"obner basis $G$ (Definition \ref{defG}) by substituting $w_{2i}$ for $w_i$ ($i=\overline{1,k}$) in all polynomials $g_M\in G$.
\end{remark}

\section{Application to immersions}

In this section we consider the (real) Grassmannians $G_{5,n}$, where $n$ is divisible by $8$. As before, $w_i\in H^{i}(G_{5,n};\mathbb{Z}_2)$, $i=\overline{1,5}$, is the $i$-th Stiefel-Whitney class of the canonical bundle $\gamma_5$ over $G_{5,n}$.

\begin{lemma}\label{immlem}
Let $n\equiv0\pmod8$ and let $\nu$ be the stable normal bundle over Grassmann manifold $G_{5,n}$. Then for the Stiefel-Whitney classes of this bundle, the following equalities hold: $w_2(\nu)=w_{1}^{2}+w_2$ and $w_i(\nu)=0$ when $i\geq 5n-4$.
\end{lemma}
\begin{proof} Let $r\geq 3$ be the integer such that $2^r<n+5\leq 2^{r+1}$. Note that this implies $n\geq 2^{r}$ since $n\equiv0\pmod8$.
In \cite[p.\ 365]{Hiller} Hiller and Stong proved that
\begin{equation}\label{totsw}
w(\nu)=w(\gamma_5\otimes\gamma_5)\cdot(1+w_1+w_2+w_3+w_4+w_5)^{2^{r+1}-n-5},
\end{equation}
and that the top nonzero class in this expression is in dimension $20+5(2^{r+1}-n-5)$. Since $n\geq 2^{r}$, we have that $20+5(2^{r+1}-n-5)\leq 20+5(2^{r}-5)=5\cdot2^{r}-5\leq 5n-5$. This proves the second equality in the statement of the lemma.

For the first one, we need the fact $w_1(\gamma_5\otimes\gamma_5)=w_2(\gamma_5\otimes\gamma_5)=0$, which is not hard to check by the method described in \cite[Problem 7-C]{MilnorSt}. Using this fact and (\ref{totsw}), one obtains that
\[w_2(\nu)={2^{r+1}-n-5\choose2}w_{1}^{2}+(2^{r+1}-n-5)w_2=w_{1}^{2}+w_2,\]
since $2^{r+1}-n-5\equiv3\pmod8$.
\end{proof}

\begin{theorem}\label{immersion}
If $n\equiv0\pmod8$, then $G_{5,n}$ immerses into $\mathbb{R}^{10n-3}$.
\end{theorem}
\begin{proof} Since $\dim G_{5,n}=5n$, in order to prove that there is an immersion of $G_{5,n}$ into $\mathbb{R}^{10n-3}$, by Hirsch's theorem (\cite[Theorem 6.4]{Hirsch}) it suffices to show that the classifying map $f_{\nu}:G_{5,n}\rightarrow BO$ of the stable normal bundle $\nu$ over $G_{5,n}$ lifts up to $BO(5n-3)$ .
\[\bfig
\morphism<600,0>[G_{5,n}`BO;f_{\nu}]
\morphism(600,500)|r|<0,-500>[BO(5n-3)`BO;p]
\morphism/-->/<600,500>[G_{5,n}`BO(5n-3);]
\efig\]

We shall use the method of modified Postnikov towers (MPT) introduced by Gitler and Mahowald in \cite{Gitler} and extended to fibrations $p:BO(m)\rightarrow BO$ for $m$ odd by Nussbaum in \cite{Nussbaum}. So, we factor out the map $p:BO(5n-3)\rightarrow BO$ as indicated in the following diagram and then we lift the map $f_{\nu}$ one level at the time. The diagram presents the $5n$-MPT for the fibration $p$ ($K_{m}$ stands for the Eilenberg-MacLane space $K(\mathbb{Z}_{2},m)$). Also, the relations that determine the $k$-invariants of the tower are listed in the table.
\[\bfig
 \morphism<800,0>[G_{5,n}`BO;f_{\nu}]
 \morphism(800,0)<900,0>[BO`K_{5n-2}\times K_{5n};k_{1}^{0}\times k_{2}^{0}]
 \morphism(800,400)<0,-400>[E_{1}`BO;]
 \morphism(800,400)<900,0>[E_{1}`K_{5n-1}\times K_{5n};k_{1}^{1}\times k_{2}^{1}]
 \morphism(800,800)<0,-400>[E_{2}`E_{1};]
 \morphism(800,800)<700,0>[E_{2}`K_{5n};k_{1}^{2}]
 \morphism(800,1200)<0,-400>[BO(5n-3)`E_{2};]
 \morphism/-->/<800,400>[G_{5,n}`E_{1};]
 \morphism/-->/<800,800>[G_{5,n}`E_{2};]
 \morphism/-->/<800,1200>[G_{5,n}`BO(5n-3);]
 \efig\]
\begin{table}[ht]
\renewcommand\arraystretch{1.5}
\noindent\[
\begin{array}{|l|}
\hline
k_{1}^{0}=w_{5n-2}\\
\hline
k_{2}^{0}=w_{5n}\\
\hline
k_{1}^{1}: (Sq^{2}+w_{2})k_{1}^{0}=0\\
\hline
k_{2}^{1}: (Sq^{2}+w_{1}^{2}+w_{2})Sq^{1}k_{1}^{0}+Sq^{1}k_{2}^{0}=0\\
\hline
k_{1}^{2}: (Sq^{2}+w_{2})k_{1}^{1}+Sq^{1}k_{2}^{1}=0\\
\hline
\end{array}
\]
\end{table}

According to Lemma \ref{immlem}, $w_{5n-2}(\nu)=w_{5n}(\nu)=0$, so, we can lift $f_{\nu}$ up to $E_1$. Also, since \[Sq^1(w_4w_5^{n-1})=(w_1w_4+w_5)w_5^{n-1}+w_4(n-1)w_1w_5^{n-1}=nw_1w_4w_5^{n-1}+w_5^{n}=w_5^{n},\]
and since $w_5^n\neq0$ in $H^{5n}(G_{5,n};\mathbb{Z}_2)\cong\mathbb{Z}_2$ (Theorem \ref{Jaworowski}), by looking at the relations in the table for $k_{2}^{1}$ and $k_{1}^{2}$, we see that it is easy to overcome these two $k$-invariants. Therefore, the only obstruction left to deal with comes from the $k$-invariant $k_{1}^{1}$. Since $H^{5n-1}(G_{5,n};\mathbb{Z}_2)\cong\mathbb{Z}_2$, it suffices to show that the map $\left(Sq^{2}+w_{2}(\nu)\right):H^{5n-3}(G_{5,n};\mathbb{Z}_2)\rightarrow H^{5n-1}(G_{5,n};\mathbb{Z}_2)$ is nontrivial. We use Lemma \ref{immlem}, formulas of Wu and Cartan and the polynomials from (\ref{k=5}) (which are trivial in $H^{*}(G_{5,n};\mathbb{Z}_2)$ since they are members of the Gr\"obner basis $G$ for the ideal $I_{5,n}$, and hence, they belong to $I_{5,n}$) to calculate
\begin{align*}
\left(Sq^{2}+w_{2}(\nu)\right)&(w_2w_5^{n-1})=(Sq^{2}+w_1^2+w_2)(w_2w_5^{n-1})\\
&=w_2^2w_5^{n-1}+(w_1w_2+w_3)(n-1)w_1w_5^{n-1}\\
&+w_2\left((n-1)w_2w_5^{n-1}+{n-1\choose2}w_1^2w_5^{n-1}\right)+w_1^2w_2w_5^{n-1}+w_2^2w_5^{n-1}\\
&=w_1^2w_2w_5^{n-1}+w_1w_3w_5^{n-1}+w_2^2w_5^{n-1}\\
&=w_2g_{(0,0,0,n-1)}+g_{(0,1,0,n-1)}+w_4w_5^{n-1}\\
&=w_4w_5^{n-1},
\end{align*}
and this class is nonzero by Theorem \ref{Jaworowski}.
\end{proof}

By the famous result of Cohen (\cite{Cohen}), Grassmanian $G_{5,n}$ can be immersed into $\mathbb{R}^{10n-\alpha(5n)}$, where $\alpha(5n)$ denotes the number of ones in the binary expansion of $5n$. This means that Theorem \ref{immersion} improves this result whenever $\alpha(5n)=2$ (and $n\equiv0\pmod8$). Such a case occurs when $n$ is a power of two, and it is known that then $G_{5,n}$ cannot be immersed into $\mathbb{R}^{10n-6}$ (\cite[p.\ 365]{Hiller}). So, if $n$ is a power of two, then for $\mathrm{imm}(G_{5,n})=\min\{d\mid G_{5,n}\mbox{ immerses into }\mathbb{R}^d\}$ the following inequalities hold
\[10n-5\leq \mathrm{imm}(G_{5,n})\leq 10n-3.\]
Actually, a sufficient and necessary condition for $\alpha(5n)=2$ and $n\equiv0\pmod8$ is that $n$ is of the form $2^r+\sum_{i=0}^s(2^{r+2+4i}+2^{r+3+4i})$, $r\geq 3$, $s\geq -1$ (where the case $s=-1$ corresponds to the case $n=2^r$).

\bibliographystyle{amsplain}

\end{document}